\documentclass[12pt,BCOR=9.25mm,paper=a4]{scrartcl}
\usepackage[english]{ayv_options}
\newcommand{\TdR}{\td\times\rd}
\title{Proximal aiming in weak KAM theory with nonsmooth Lagrangian}
\author{Yurii Averboukh}\address{Krasovskii Institute of Mathematics and Mechanics, \\ & Yekaterinburg, Russia; \\ & HSE University, Moscow, Russia}
\email{ayv@imm.uran.ru} 
\date{}

\begin{document}
	\maketitle
	\begin{abstract}
		This work extends weak KAM theory to the case of a nonsmooth Lagrangian satisfying a superlinear growth condition. Using the solution of a weak KAM equation that is a stationary Hamilton-Jacobi equation and the proximal aiming method, we construct a family of discontinuous feedback strategies that are nearly optimal for every time interval. This result leads to an analogue of the weak KAM theorem. Additionally, as in classical weak KAM theory, we demonstrate that the effective Hamiltonian (Ma\~{n}\'{e} critical value) can be determined by solving a linear programming problem in the class of probability measures.
		\msccode{49J52, 49J05, 35F21, 90C05}
	\end{abstract}
	
	\section{Introduction}
	This paper focuses on constructing feedback strategies for a calculus of variations problem on a $d$-dimensional torus, where the time interval is sufficiently large and the Lagrangian is nonsmooth.
	
	Let us briefly outline the problem under consideration. We examine settings of weak KAM theory \cite{Fathi2012,Sorrentino,Biryuk2010,Mane1992,Mather1991}, which seeks a function $\varphi$ and a constant $\overline{H}$ such that, for each $r>0$, the value function in the minimization problem for the functional
	\begin{equation}\label{intrd:func:L_varphi}\int_0^rL(x(t),\dot{x}(t))dt+\varphi(x(r))\end{equation} equals $\varphi-\overline{H} r$. Here, the function $L$ is called the Lagrangian, and the quantity $\overline{H}$ is referred to as the effective Hamiltonian or the Ma\~{n}\'{e} critical value. In particular, weak KAM theory implies that for large $r$, the solution to the problem
	\begin{equation}\label{intrd:intro:payoff}\text{minimize }\frac{1}{r}\int_0^rL(x(t),\dot{x}(t))dt\end{equation} converges to the constant $-\overline{H}$.
	
	Typically, weak KAM theory \cite{Fathi2012,Sorrentino,Biryuk2010} considers a smooth Lagrangian defined on the tangent bundle of a finite-dimensional compact manifold satisfying Tonelli conditions. Moreover (see \cite{Evans2008,Evans2004,Bernard2012,Biryuk2010,Fathi2007}), the pair $(\varphi,\overline{H})$ satisfies the weak KAM equation:
	\begin{equation}\label{intrd:eq:HJ}H(x,-\nabla \varphi)=\overline{H},\end{equation} where the unknowns are a continuous function $\varphi$ and a constant $\overline{H}$, and the Hamiltonian $H$ is defined via the Legendre transform:
	\begin{equation}\label{intrd:intro:H}H(x,p)\triangleq \max_{v\in\rd}\Big[pv-L(x,v)\Big].\end{equation}
	Additionally, the constant $-\overline{H}$ equals the value in the minimization problem of the functional
	\[\int_{\td\times\rd}L(x,v)\mu(d(x,v)),\] over the set of all holonomic measures or measures invariant under the flow generated by the Euler-Lagrange equation
	\[\frac{d}{dt}L_v(x,v)=L_x(x,v).\] A measure providing the minimum is called a Mather measure.
	Although the function $\varphi$ is generally nonsmooth and satisfies the Hamilton-Jacobi equation in the viscosity sense, it has a derivative almost everywhere with respect to the projection of the Mather measure onto the configuration space. Thus, the continuous feedback strategy $\mathbbm{v}:\td\rightarrow\rd$ satisfying on the support of the projection of the Mather measure the equality
	\[\mathbbm{v}(x)=H_p(x,-\nabla \varphi(x)),\] is optimal.
	
	These results heavily rely on the smoothness of the Lagrangian (and, consequently, the Hamiltonian). Several generalizations exist. For instance, the case of a noncompact manifold is studied in \cite{Fathi2012}. Additionally, \cite{Gomes2014} explores weak KAM theory for piecewise smooth Lagrangians.
	
	A notable feature of weak KAM theory is that the existence of a viscosity solution to the weak KAM equation~\eqref{intrd:eq:HJ} which is crucial for constructing optimal strategies and characterizing of $\overline{H}$ can be established without the smoothness assumption on the Hamiltonian (see \cite{lions_et_al}). This raises the question of extending weak KAM theory to the case of essentially nonsmooth Lagrangians. This paper addresses this question partially. Our methodology utilizes the proximal aiming technique presented in \cite{Clarke1999}. This approach extends the general extremal shift method initially formulated for differential games (see \cite{NN_PDG}). The key benefit of proximal aiming lies in its capacity to generate nearly optimal strategies while ensuring only linear growth of the approximation error over time.
	
	We focus on the simplest compact manifold that is the $d$-dimensional torus $\td$. Our main result establishes that a discontinuous strategy realizing the proximal aiming (i.e., the extremal shift rule with directions determined by the Moreau-Yosida transform) yields nearly optimal solutions for the calculus of variations problem with criterion \eqref{intrd:intro:payoff} as the time interval tends to infinity. Furthermore, we extend the Mather measure technique and demonstrate that the Ma\~{n}\'{e} critical value can be characterized by minimizing the action of probability measures on the Lagrangian over the class of holonomic probability measures.
	
	The paper is organized as follows. Section~\ref{sect:notation} discusses the phase space, admissible trajectories and concepts of nonsmooth analysis including the definition of viscosity solutions to equation~\eqref{intrd:eq:HJ}. We also recall the Moreau-Yosida transform and its properties which are key to subsequent constructions. Section~\ref{sect:feedback} introduces the proximal aiming method and demonstrates its application in constructing strategies that guarantee a value of the functional~\eqref{intrd:func:L_varphi} no less than $\varphi(y)-\overline{H}r$  on $[0,r]$ with arbitrary accuracy. Here $(\varphi,\overline{H})$ is a supersolution to equation~\eqref{intrd:eq:HJ}. Section~\ref{sect:lower_bound} shows that, if $(\varphi,\overline{H})$ is a subsolution to equation~\eqref{intrd:eq:HJ}, the value of the payoff functional~\eqref{intrd:func:L_varphi} at the initial point $y$ does not exceed $\varphi(y)-\overline{H}r$. The implications of these estimates are summarized in Section~\ref{sect:weak_KAM}. Finally, Section~\ref{sect:mather} discusses the characterization of the Ma\~{n}\'{e} critical value $\overline{H}$ via the Mather measure.
	
	\section{Definitions and notation}\label{sect:notation}
	\subsection{Phase space}
	Let $\td$ denote the $d$-dimensional torus (i.e., $\td\triangleq\rd/\mathbb{Z}^d$) which serves as the configuration space. The elements of $\td$ are sets of the form
	\[x=\{x'+n:\, n\in \mathbb{Z}^d\},\] where $x'$ is an element of $\rd$. Below, we assume that $\rd$ consists of column vectors, and the norm on $\rd$ is denoted by $|\cdot|$.
	One can add a vector in $\rd$ to an element of $\td$. Specifically, if $x\in\td$ and $v\in\rd$,
	\[x+v=\{x'+n:\,n\in\mathbb{Z}^d\},\] where $x'$ is any element of $x$.
	
	The distance between $x,y\in\td$ is defined as:
	\[\rho_{\td}(x,y)\triangleq \inf\big\{|y'-x'|:\, x'\in x,\, y'\in y\big\}.\] Equivalently, the distance can be computed as:
	\[\rho_{\td}(x,y)=\inf\big\{|v|:\, v\in\rd,\, x+v=y\big\}.\] The tangent space to $\td$ is naturally $\rd$, and the cotangent space is $\rds$ consisting of $d$-dimensional row vectors. The distance on $\rds$ is also denoted by $|\cdot|$. For $p\in\rds$ and $v\in\rd$, $pv$ represents the matrix multiplication of a row and a column. The transpose operation is denoted by $\top$. The quantity $\mathbbm{1}_{P}$ equals $1$ or $0$ depending on the truth value of the logical expression $P$.
	Finally, $\mathbbm{B}_\alpha$ stands for the ball in $\rd$ of radius $\alpha$ centered at the origin.
	
	\subsection{Admissible trajectories}
	We consider absolutely continuous curves $x(\cdot):[0,r]\rightarrow\td$. By \cite[Theorems 2.14, 2.17]{Buttazzo1998}, each such $x(\cdot)$ is continuous. We interpret $\dot{x}(t)$ as the velocity at time~$t$. Furthermore, it is convenient to describe absolutely continuous curves  in terms of controlled processes. Here, a controlled process on $[0,r]$ is a pair $(x(\cdot),v(\cdot))$ such that
	\begin{itemize}
		\item $v(\cdot)\in L^1([0,r];\rd)$;
		\item $x(\cdot)\in C([0,r],\td)$, and, for all $s\in [0,r]$,
		\[x(s)-x(0)=\int_0^sv(t)dt.\]
	\end{itemize}
	
	\subsection{Nonsmooth analysis and viscosity solutions}
	We adopt an approach to viscosity solutions based on nonsmooth analysis \cite{Subbotin}. It is equivalent to one involving  test functions \cite{Subbotin,Bardi2008c}. First, we introduce the definitions of sub- and superdifferentials following \cite{Penot2013,Clarke1998a}.
	\begin{definition}\label{notation:def:subsuper_diff} Let $\phi:\td\rightarrow\mathbb{R}$ and $x\in\td$. The set 
		\[\partial_D^-\phi(x)\triangleq \Bigg\{p\in\rds:\,\liminf_{h\downarrow 0,\, w'\rightarrow w}\frac{\phi(x+hw')-\phi(x)}{h}\geq pw\text{ for all }w\in\rd\Bigg\}\] is called the Hadamard subdifferential.
		
		Similarly, the set \[\partial_D^+\phi(x)\triangleq \Bigg\{p\in\rds:\,\limsup_{h\downarrow 0,\, w'\rightarrow w}\frac{\phi(x+hw')-\phi(x)}{h}\leq pw\text{ for all }w\in\rd\Bigg\}\] is called the Hadamard superdifferential.
	\end{definition}
	We also use proximal sub- and superdifferentials \cite{Clarke1999,Clarke1998a,Penot2013}.
	\begin{definition}\label{notation:def:proximal} A vector $p\in\rds$ is a proximal subgradient of $\phi$ at $x$ if there exist $\alpha>0$ and $\sigma\geq 0$ such that, for all $w\in\mathbbm{B}_\alpha$,
		\[\phi(x+w)-\phi(x)\geq pw-\sigma|w|^2.\] The set of all proximal subgradients is the proximal subdifferential of $\phi$ at $x$. It is denoted by $\partial_P^-\phi(x)$.
		
		The proximal superdifferential of $\phi$ at $x$   consists of all $p\in\rds$ such that for some $\alpha>0$ and $\sigma\geq 0$,
		\[\phi(x+w)-\phi(x)\leq pw+\sigma|w|^2\] holds for all $w\in \mathbbm{B}_\alpha$. We denote the proximal subdifferential of $\phi$ at $x$  by $\partial_P^+\phi(x)$. The elements of $\partial_P^+\phi(x)$ are called supergradients of $\phi$ at $x$.
	\end{definition}
	Notice that 
	\begin{equation}\label{notation:intro:diffs}
		\partial_P^-\phi(x)\subset\partial_D^-\phi(x),\, \partial_P^+\phi(x)\subset\partial_D^+\phi(x).
	\end{equation}
	
	\begin{definition}\label{notation:def:subsolution} A pair $(\varphi,\overline{H})$, where $\varphi:\td\rightarrow \mathbb{R}$ is upper semicontinuous and $\overline{H}\in\mathbb{R}$, is a subsolution to equation~\eqref{intrd:eq:HJ} if, for all $x\in\td$ and $p\in \partial_D^+\phi(x)$,
		\[H(x,-p)\leq \overline{H}.\]
		
		A pair $(\varphi,\overline{H})$, where $\varphi:\td\rightarrow \mathbb{R}$ is lower semicontinuous and $\overline{H}\in\rd$, is a supersolution to equation~\eqref{intrd:eq:HJ} if, for all $x\in\td$ and $p\in \partial_D^-\phi(x)$,
		\[H(x,-p)\geq \overline{H}.\]
		
		A pair $(\phi,\overline{H})$, where $\phi:\td\rightarrow\mathbb{R}$ is continuous and $\overline{H}\in \mathbb{R}$, is a viscosity solution to equation~\eqref{intrd:eq:HJ} provided that it is both a sub- and supersolution.
	\end{definition}
	
	If $H$ is continuous and satisfies the coercivity condition
	\begin{equation}\label{notation:cond:H_coercitivity}
		\lim_{|p|\rightarrow\infty} \inf_{x\in \td}\frac{H(x,p)}{|p|}=+\infty,
	\end{equation} a viscosity solution to equation~\eqref{intrd:eq:HJ} exists (see \cite[Theorem~1]{lions_et_al}, \cite[Theorem 1.30]{Tran}). Moreover, $\overline{H}$ is uniquely determined. Additionally, \cite[Proposition II.4.1]{Bardi2008c} shows that, if $(\varphi,\overline{H})$ is a viscosity solution to~\eqref{intrd:eq:HJ}, then $\varphi$ is Lipschitz continuous with a constant determined only by $H$.
	
	\subsection{Moreau-Yosida transform}
	The Moreau-Yosida transform plays a key role in our constructions. It was studied in \cite{Clarke1999} for $\rd$. For the sake of completeness, we discuss its properties for the torus.
	\begin{definition}\label{notation:intrd:IM} Let $\phi:\td\rightarrow\mathbb{R}$ and $\varkappa>0$. The Moreau-Yosida transform of $\phi$ is the function $\phi_\varkappa$ defined by the following rule:
		\[\phi_\varkappa(x)\triangleq\inf\Bigg\{\phi(x+v)+\frac{1}{2\varkappa}|v|^2:\, v\in\rd\Bigg\}.\] Equivalently,
		\[\phi_\varkappa(x)= \inf\Bigg\{\phi(y)+\frac{1}{2\varkappa}\rho_{\td}^2(x,y):\, y\in\td\Bigg\}.\]
	\end{definition}
	Let $b_\varkappa[x]$ denote an element of $\rd$ such that 
	\begin{equation}\label{notation:intro:y_varkapp}
		\phi(x+b_\varkappa[x])+\frac{1}{2\varkappa}\big|b_\varkappa[x]\big|^2=\phi_\varkappa(x).
	\end{equation}
	
	\begin{proposition}\label{notation:prop:y_varkappa} Assume that $\phi$ is Lipschitz continuous, then 
		\[\big|b_\varkappa[x]\big|\leq C_1\varkappa,\] where $C_1=2K$.
	\end{proposition} 
	\begin{proof}
	Notice that 
		\[\frac{1}{2\varkappa}\big|b_\varkappa[x]\big|^2\leq \phi(x)-\phi(x+b_\varkappa[x])\leq K\big|b_\varkappa[x]\big|.\] This gives the statement of the proposition.		
	\end{proof}
	
	\begin{corollary}\label{notation:corollary:vicinity} Let $\phi$ be Lipschitz continuous with constant $K$. Then, for all $x\in \td$,
		\[|\phi(x)-\phi_\varkappa(x)|\leq C_2\varkappa,\] where $C_2$ depends only on $\phi$.
	\end{corollary}
	\begin{proof}
		From the Lipschitz continuity of $\phi$ and the definition of the vector $b_\varkappa[x]$,
		\[
	|\phi(x)-\phi_\varkappa(x)|\leq |\phi(x)-\phi(x+b_\varkappa[x])|+\frac{1}{2\varkappa}\big|b_\varkappa[x]\big|^2.\]
		Using Proposition~\ref{notation:prop:y_varkappa}, the right-hand side is bounded by
	\[KC_1\varkappa+\frac{1}{2\varkappa}C_1\varkappa^2.\] From this the corollary follows directly.
	\end{proof}
	
	\begin{proposition}\label{notation:prop:Lips_IM} If $\phi$ is Lipschitz continuous with constant $K$, then, for every $\varkappa>0$, $\phi_\varkappa$ is also Lipschitz continuous with the same constant.
	\end{proposition}
	\begin{proof}
		Let $x,y\in \td$. Note that 
		\[\phi_\varkappa(y)\leq \phi\big(y+b_\varkappa[x]\big)+\frac{1}{2\varkappa}\big|b_\varkappa[x]\big|^2.\] Hence,
		\[\phi_\varkappa(y)-\phi_\varkappa(x)\leq \phi\big(y+b_\varkappa[x]\big)-\phi\big(x+b_\varkappa[x]\big)\leq K\rho_{\td}(x,y).\] Swapping $x$ and $y$ yields $\phi_\varkappa(x)-\phi_\varkappa(y)\leq K\rho_{\td}(x,y)$.
	\end{proof}
	
	The following is a variant of \cite[Lemma 3.1]{Clarke1999} for the torus.
	\begin{proposition}\label{notation:prop:prox_diff} For all $x\in \td$, 
		\begin{equation}\label{notation:intro:p_varkappa}p_\varkappa[x]\triangleq \frac{1}{\varkappa}(-b_\varkappa[x])^\top\in\partial_P^-\phi(x+b_\varkappa[x]).\end{equation}
	\end{proposition}
	\begin{proof}
		Let $w\in\rd$. By the definition of the Moreau-Yosida transform,
		\[\phi(x+b_\varkappa[x]+w)+ \frac{1}{2\varkappa}\big|b_\varkappa[x]+w\big|^2\geq \phi(x+b_\varkappa[x])+\frac{1}{2\varkappa}\big|b_\varkappa[x]\big|^2.\] Rearranging,
		\[\begin{split}
			\phi(x+b_\varkappa[x]+w)-\phi(x+b_\varkappa[x])&\geq \frac{1}{2\varkappa}\big|b_\varkappa[x]\big|^2-\frac{1}{2\varkappa}\big|b_\varkappa[x]+w\big|^2\\ &=
			\frac{1}{\varkappa}(-b_\varkappa[x])^\top w-\frac{1}{2\varkappa}|w|^2.\end{split}\] Thus, the proposition holds.
	\end{proof}
	
	\begin{corollary}\label{notation:corollary:bound_p} If $\phi$ is Lipschitz continuous with constant $K$, then, for all $\varkappa>0$ and $x\in \td$, 
		\[\big|p_\varkappa[x]\big|\leq K.\]
	\end{corollary}
	This follows from Proposition~\ref{notation:prop:prox_diff} and the fact that subdifferentials of Lipschitz functions are bounded by the Lipschitz constant \cite[Proposition 1.5, Theorem 6.1]{Clarke1998a}.
	
	Finally, we establish a Taylor-like expansion, analogous to \cite[Lemma 3.5]{Clarke1999}.
	\begin{proposition}\label{notation:prop:taylor}
		For all $x\in\td$ and $v\in\rd$,
		\[\phi_\varkappa(x+w)\leq \phi_\varkappa(x)+p_\varkappa[x] w+\frac{1}{2\varkappa}|w|^2,\] where $p_\varkappa[x]$ is defined in Proposition~\ref{notation:prop:prox_diff}.
	\end{proposition}
	\begin{proof}
		By the definition of $\phi_\varkappa$, one has that, for all $w\in\rd$,
		\[\phi_\varkappa(x+w)\leq \phi(x+w+(b_\varkappa[x]-w))+\frac{1}{2\varkappa}\big|b_\varkappa[x]-w\big|^2.\] Expanding the last term and using the definitions of $b_\varkappa[x]$ and $p_\varkappa[x]$, we obtain the required inequality.
	\end{proof}
	
	\section{Feedback strategies and upper estimate}\label{sect:feedback}
	We assume that the Lagrangian $L$ is continuous and satisfies the superlinear growth condition:
	\begin{equation}\label{feedback:cond:L}
		\lim_{|v|\rightarrow\infty} \inf_{x\in\td}\frac{L(x,v)}{|v|}=+\infty.
	\end{equation}
	Then, the Hamiltonian $H$ defined by~\eqref{intrd:intro:H} is continuous and satisfies superlinear growth condition~\eqref{notation:cond:H_coercitivity}.
	
	For $\delta>0$ and $A\geq 0$, let
	\begin{equation}\label{feedback:intro:omega}\omega_A(\delta)\triangleq \sup\Big\{|L(x,v)-L(y,v)|:\, x,y\in\td,\,\rho_{\td}(x,y)\leq \delta,\, v\in\rd,\, |v|\leq A \Big\}.\end{equation} For fixed $A$, the function $\omega_A(\cdot)$ is a modulus of continuity of $L$ with respect to the first variable on $\td\times \mathbbm{B}_A$.
	
	Within this section, we assume that we are given with a supersolution to~\eqref{intrd:eq:HJ} that is a pair $(\varphi,\overline{H})$; moreover the function $\varphi$ is Lipschitz continuous.
	
	For $\varkappa>0$, let us consider the Moreau-Yosida transform. For each $x\in \td$, we choose
	\begin{equation}\label{feedback:intro:v_varkappa}\mathbbm{v}_\varkappa[x]\in \underset{v\in \rd}{\operatorname{Argmax}}\Big[-p_\varkappa[x]v-L(x+b_\varkappa[x],v)\Big],\end{equation} where $b_\varkappa[x]$ and $p_\varkappa[x]$ are defined by~\eqref{notation:intro:y_varkapp} and~\eqref{notation:intro:p_varkappa} respectively.
	
	\begin{proposition}\label{feedback:prop:bound} There exists a constant $C_3$ such that, for all $\varkappa>0$ and $x\in\td$,
		\[\big|\mathbbm{v}_\varkappa[x]\big|\leq C_3.\]
	\end{proposition}
	\begin{proof}
		The value 
		\[\max_{v\in\rd}\Big[-p_\varkappa[x]v-L(x+b_\varkappa[x],v)\Big]\] is bounded below by $-C_1'$, where \[C_1'\triangleq  \sup_{y\in\td}|L(y,0)|.\]  Corollary~\ref{notation:corollary:bound_p} says that $\big|p_\varkappa[x]\big|\leq K$, where $K$ is the Lipschitz constant for $\varphi$. Due to the superlinearity of $L$, there exists $c\geq 0$ such that
		\[L(x,v)\geq (K+C_1'+1)|v|\text{ whenever }|v|\geq c.\] Thus, for $|v|\geq c\vee 1$,
		\[\begin{split}
			-p_\varkappa[x]v-L(x+b_\varkappa[x],v)&\leq K|v|-(K+C_1'+1)|v|\\&=-(C_1'+1)|v|<-C_1'.\end{split}\] Setting $C_3\triangleq c\vee 1$ completes the proof.
	\end{proof}
	
	The strategy $\mathbbm{v}_\varkappa$ may be discontinuous. We employ the step-by-step scheme proposed by  Krasovskii and  Subbotin (see \cite{NN_PDG}).
	Let 
	\begin{itemize}
	\item	$r$ be a length of the time interval;
	\item $\Delta=\{t_i\}_{i=1}^n$ be a partition of $[0,r]$; 
	\item $d(\Delta)$ denote the fineness of $\Delta$, i.e.,
	\[d(\Delta)\triangleq \max_{i=0,\ldots,n-1}(t_{i+1}-t_i);\] 
	\item  $y\in \td$. \end{itemize} We define the motion $x_{y,\varkappa,r,\Delta}(\cdot)$ and control $v_{y,\varkappa,r,\Delta}(\cdot)$ sequentially on each time interval $[t_i,t_{i+1}]$ by the rule:
	\begin{itemize}
		\item $x_{y,\varkappa,r,\Delta}(0)\triangleq y$;
		\item if the motion is defined up to $t_i$, then, for $t\in [t_i,t_{i+1})$,
		\[v_{y,\varkappa,r,\Delta}(t)\triangleq \mathbbm{v}_\varkappa[x_{y,\varkappa,r,\Delta}(t_i)],\] and, for $t\in [t_i,t_{i+1}]$,
		\[x_{y,\varkappa,r,\Delta}(t)\triangleq x_{y,\varkappa,r,\Delta}(t_i)+(t-t_i)\mathbbm{v}_\varkappa[x_{y,\varkappa,r,\Delta}(t_i)].\]
	\end{itemize}
	
	\begin{theorem}\label{feedback:th:main} Let $(\varphi,\overline{H})$, where $\varphi:\td\rightarrow\mathbb{R}$ is Lipschitz continuous, be a supersolution to~\eqref{intrd:eq:HJ}. Given $\varepsilon>0$, there exist $\varkappa_0>0$ and a function $\delta:(0,\varkappa_0]\rightarrow (0,+\infty)$ such that, for all $r>0$, $\varkappa\in (0,\varkappa_0]$, partitions $\Delta$ of $[0,r]$ with $d(\Delta)\leq\delta(\varkappa)$, and each $y\in\td$, the following inequality holds true:
		\[\varphi(x_{y,\varkappa,r,\Delta}(r))+\int_{0}^rL(x_{y,\varkappa,r,\Delta}(t),v_{y,\varkappa,r,\Delta}(t))dt\leq \varphi(y)- \overline{H}r+(r+1)\varepsilon.\]
	\end{theorem}
	\begin{proof} For brevity, let:
		\begin{itemize}
			\item $x_i\triangleq x_{y,\varkappa,r,\Delta}(t_i)$;
			\item $v_i\triangleq v_{y,\varkappa,r,\Delta}(t_i)$;
			\item $p_i\triangleq p_\varkappa[x_i]$;
			\item $y_i\triangleq x_i+ b_\varkappa[x_i]$.
		\end{itemize} By Proposition~\ref{feedback:prop:bound}, $|v_i|\leq C_3$.
		
		For each $i=0,\ldots,n-1$, let us consider the difference
		\[\varphi_\varkappa(x_{i+1})-\varphi_\varkappa(x_i).\] Note that $x_{i+1}=x_i+(t_{i+1}-t_i)v_i$. Due to Proposition~\ref{notation:prop:taylor},
		\begin{equation}\label{feedback:ineq:varphi_varkappa}\varphi_\varkappa(x_{i+1})-\varphi_\varkappa(x_i)\leq (t_{i+1}-t_i)p_iv_i+\frac{(t_{i+1}-t_i)^2}{2\varkappa}C_3^2.\end{equation} Next, consider the integral
		\[\int_{t_i}^{t_{i+1}}L\big(x_{y,\varkappa,r,\Delta}(t),v_{y,\varkappa,r,\Delta}(t)\big)dt=\int_{t_i}^{t_{i+1}}L\big(x_i+v_it,v_i\big)dt.\] Using the modulus of continuity $\omega_{C_3}(\cdot)$, we deduce
		\[\Big|L\big(x_i+v_it,v_i\big)-L(y_i,v_i)\Big|\leq\omega_{C_3}(\rho_{\td}(x_i+v_it,x_i))+\omega_{C_3}(\rho_{\td}(x_i,y_i)).\] The first term is bounded by $\omega_{C_3}(d(\Delta)C_3)$, while the second, thanks to Proposition~\ref{notation:prop:y_varkappa}, is bounded by $\omega_{C_3}(C_1\varkappa)$. Thus,
		\[\begin{split}
			\Bigg|\int_{t_i}^{t_{i+1}}L\big(x_{y,\varkappa,r,\Delta}(t),v_{y,\varkappa,r,\Delta}(t)\big)dt&- (t_{i+1}-t_i)L(y_i,v_i)\Bigg|\\\leq (t_{i+1}-t_i)&[\omega_{C_3}(d(\Delta)C_3)+\omega_{C_3}(C_1\varkappa)].\end{split}\] Combining this with inequality~\eqref{feedback:ineq:varphi_varkappa}, we arrive at the estimate:
		\[
		\begin{split}
			\varphi_\varkappa(x_{i+1})-\varphi_\varkappa&(x_i)+\int_{t_i}^{t_{i+1}}L\big(x_{y,\varkappa,r,\Delta}(t),v_{y,\varkappa,r,\Delta}(t)\big)dt \\\leq (t_{i+1}&-t_i)\Big[p_iv_i+L(y_i,v_i)\Big]\\&+\frac{(t_{i+1}-t_i)^2}{2\varkappa}C_3^2+(t_{i+1}-t_i)\Big[\omega_{C_3}(d(\Delta)C_3)+\omega_{C_3}(C_1\varkappa)\Big].
		\end{split}
		\] By Proposition~\ref{notation:prop:prox_diff} and inclusion~\eqref{notation:intro:diffs}, $p_i\in \partial_P^-\varphi(y_i)\subset \partial_D^-\varphi(y_i)$. Since $(\varphi,\overline{H})$ is a supersolution to~\eqref{intrd:eq:HJ}, whereas $v_i$ satisfies
		\[p_iv_i+L(y_i,v_i)=\min_{v\in\rd}\Big[p_iv+L(y_i,v)\Big],\] we have that
		\[p_iv_i+L(y_i,v_i)\leq -\overline{H}.\] Hence,
		\[
		\begin{split}
			\varphi_\varkappa(x_{i+1})-\varphi_\varkappa(x_{i})+	\int_{t_i}^{t_{i+1}}L\big(x_{y,\varkappa,r,\Delta}(t),v_{y,\varkappa,r,\Delta}(t)\big)&dt  \leq -(t_{i+1}-t_i)\overline{H}+  \\+\frac{(t_{i+1}-t_i)^2}{2\varkappa}C_3^2+(t_{i+1}-t_i)\Big[\omega_{C_3}&(d(\Delta)C_3)+\omega_{C_3}(C_1\varkappa)\Big].
		\end{split}
		\] Summing over $i=0,\ldots,n-1$ yields
		\[
		\begin{split}
			\varphi_\varkappa(x_{n})+	\int_{0}^{r}L\big(x_{y,\varkappa,r,\Delta}(&t),v_{y,\varkappa,r,\Delta}(t)\big)dt \leq -r\overline{H}+ \varphi_\varkappa(y) \\&+\frac{rd(\Delta)}{2\varkappa}C_3^2+r\Big[\omega_{C_3}(d(\Delta)C_3)+\omega_{C_3}(C_1\varkappa)\Big].
		\end{split}
		\] Using Corollary~\ref{notation:corollary:vicinity} and the inequality $\varphi_\varkappa(y)\leq \varphi(y)$, we conclude that
		\[
		\begin{split}
			\varphi(x_{n})+	\int_{0}^{r}L\big(x_{y,\varkappa,r,\Delta}(&t),v_{y,\varkappa,r,\Delta}(t)\big)dt \leq -r\overline{H}+ \varphi(y) \\&+\frac{rd(\Delta)}{2\varkappa}C_3^2+r\omega_{C_3}(d(\Delta)C_3)+r\omega_{C_3}(C_1\varkappa)+C_2\varkappa.
		\end{split}
		\] 
		Now, we choose $\varkappa_0$ such that
		\[C_2\varkappa_0\leq \varepsilon,\, \omega_{C_3}(C_1\varkappa_0)\leq \varepsilon/2.\] For fixed $\varkappa>0$, let $\delta(\varkappa)$ be such that
		\[\delta(\varkappa)\leq \varepsilon\varkappa/(2C_3^2),\, \omega_{C_3}(\delta(\varkappa)C_3)\leq \varepsilon/4.\] Thus the statement of the theorem follows.
	\end{proof}
	
	\section{Lower estimate}\label{sect:lower_bound}
	\begin{theorem}\label{lower:th:main} Let $(\varphi,\overline{H})$ be a subsolution to~\eqref{intrd:eq:HJ}. Additionally, assume that  $\varphi:\td\rightarrow\mathbb{R}$ is Lipschitz continuous.  Then, for every $r>0$ and each controlled process $(x(\cdot),v(\cdot))$ on $[0,r]$,
		\[\varphi(x(r))+\int_0^rL(x(t),v(t))dt\geq \varphi(x(0))-\overline{H}r.\]
	\end{theorem}

The proof of this statement relies on the upper Moreau-Yosida transform of a function $\phi:\td\rightarrow\mathbb{R}$ introduced in the following way:
\[\phi^\varkappa(x)\triangleq \sup\Bigg\{\phi(x+v)-\frac{1}{2\varkappa}|v|^2:\, v\in\rd\Bigg\}.\]
Note that $\phi^\varkappa(x)=-(-\phi)_\varkappa(x)$. This give that, if $\phi$ is Lipschitz continuous with constant $K$, and $b^\varkappa[x]$ satisfies the equality:
\[\phi^\varkappa(x)= \phi(x+b^\varkappa[x])-\frac{1}{2\varkappa}\big|b^\varkappa[x]\big|^2,\]
then the following properties hold:
\begin{enumerate}[label=(K\arabic*)]
    \item\label{lower:prop:lip} $\phi^\varkappa$ is Lipschitz continuous with the constant $K$;
\item\label{lower:prop:b_bound} $\big|b^\varkappa[x]\big|\leq C_1\varkappa$;
\item\label{lower:prop:dist} $\|\phi^\varkappa-\phi\|_{C(\td)}\leq C_2\varkappa$;
\item\label{lower:prop:superdiff} $p^\varkappa[x]\triangleq\varkappa^{-2}(b^\varkappa[x])^\top\in\partial_P^+\phi(x+b^\varkappa[x])$;
\item\label{lower:prop:p_bound} $|p^\varkappa[x]|\leq K$;
\item\label{lower:prop:taylor} for all $w\in \rd$, one has the inequality
	\begin{equation}\label{lower:ineq:phi_kappa_sup}
		\phi^\varkappa(x+w)-\phi(x)\geq p^\varkappa[x]w-\frac{1}{2\varkappa}|w|^2.
	\end{equation}
\end{enumerate}

\begin{proof}[Proof of Theorem~\ref{lower:th:main}]
	The proof is divided into five steps for clarity.
	
	\begin{enumerate}[wide, labelwidth=!, labelindent=0pt, label=\textit{Step \arabic*}.]
		\item Let $K$ stand for the Lipschitz constant of the function $\varphi$.  Choose $A$ such that, for all $(x,v)\in \td\times\rd$ satisfying $|v|\geq A$:
		\begin{equation}\label{lower:ineq:A}L(x,v)\geq K|v|-\overline{H}.\end{equation}
		This is possible due to the Lagrangian's superlinearity. Similarly, choose $B$ satisfying
		\begin{equation}\label{lower:ineq:B}
			L(x,v)\geq (K+1)|v|\text{ whenever }|v|\geq B.
		\end{equation}
		Furthermore, recall that $\omega_A(\cdot)$ stands for the modulus of continuity of $L$ with respect to the first variable on $\td\times \mathbbm{B}_A$. Finally, notice that $L(x,v)\geq C_2'$ for some constant. Without loss of generality one can assume that $C_2'\leq 0$.
		
		\item For $\varepsilon>0$, we choose $\varkappa>0$ such that
		\begin{equation}\label{lower:ineq:varkappa_choice}
			C_2\varkappa<\varepsilon/2,\ \ \omega_A(C_1\varkappa)\leq \varepsilon/3.
		\end{equation}
		
		Let us define the set
		\[\mathscr{S}_{\varepsilon,\varkappa}\triangleq \Bigg\{s\in [0,r]:\,\varphi^\varkappa(x(s))+\int_0^sL(x(t),v(t))dt\geq \varphi^\varkappa(x(0))-\big(\overline{H}+\varepsilon\big) s\Bigg\}.\]
		Notice that $\mathscr{S}_{\varepsilon,\varkappa}$ is nonempty. Indeed, $0\in\mathscr{S}_{\varepsilon,\varkappa}$.
		
		\item Let us show that $s_*\triangleq \sup\mathscr{S}_{\varepsilon,\varkappa}$ lies in $\in\mathscr{S}_{\varepsilon,\varkappa}$. Let $\{s_n\}_{n=1}^\infty\subset\mathscr{S}_{\varepsilon,\varkappa}$ be an increasing sequence converging to $s_*$. Then,
		\[\varphi^\varkappa(x(s_n))+\int_0^{s_n}L(x(t),v(t))dt\geq \varphi^\varkappa(x(0))-\big(\overline{H}+\varepsilon\big) s_n.\]
		Since $L(x(t),v(t))\geq C_2'$ and $s_n\leq s_*$, we have that
		\[\varphi^\varkappa(x(s_n))+\int_0^{s_*}L(x(t),v(t))dt\geq \varphi^\varkappa(x(0))-\big(\overline{H}+\varepsilon\big) s_n+C_2'(s_*-s_n).\]
		Taking the limit as $n\rightarrow \infty$ and using the continuity of $\varphi^\varkappa$, we conclude that $s_*\in \mathscr{S}_{\varepsilon,\varkappa}$.
		
		\item For $s\in\mathscr{S}_{\varepsilon,\varkappa}$ with $s<r$, we shall prove that there exists $s_+\in (s,r]$ lying in $\mathscr{S}_{\varepsilon,\varkappa}$. We consider two cases.
		\begin{enumerate}
			\item There exists $\tau\in (0,r-s]$ such that
			\begin{equation}\label{lower:ineq:prenorm}
				\int_s^{s+\tau}\Big(|v(t)|\mathbbm{1}_{|v(t)|\geq B}-K|v(t)|\mathbbm{1}_{|v(t)|<B}\Big)dt\geq \Big(-\overline{H} -C_2'\Big)\tau.
			\end{equation} Here, $C_2'$ is a lower bound of the Lagrangian $L$. Additionally, $C_2'\leq 0$.
			
			In this case, we set $s_+=s+\tau$.  Then,
			\begin{equation}\label{lower:ineq:s_plus_0_full}
				\begin{split}
					\varphi^\varkappa(x(s_+))-\varphi^\varkappa(x(0)&)+\int_{0}^{s_+}L(x(t),v(t))dt\\=
					\varphi^\varkappa(x(s_+)&)-\varphi^\varkappa(x(s))+\int_{s}^{s_+}L(x(t),v(t))dt\\&+\varphi^\varkappa(x(s))-\varphi^\varkappa(x(0))+\int_{0}^{s}L(x(t),v(t))dt.
				\end{split}
			\end{equation}
			Since $s\in\mathscr{S}_{\varepsilon,\varkappa}$, we have that
			\begin{equation}\label{lower:ineq:s_0}
				\varphi^\varkappa(x(s))-\varphi^\varkappa(x(0))+\int_{0}^{s}L(x(t),v(t))dt\geq -\big(\overline{H}+\varepsilon\big) s.
			\end{equation}
			Due to the Lipschitz continuity of the function $\varphi^\varkappa$ (see property~\ref{lower:prop:lip}) and since $x(s_+) = x(s) + \int_{s}^{s_+}v(t)dt$, we derive the inequality
			\begin{equation}\label{lower:ineq:s_s_plus_1}
				\begin{split}
					\varphi^\varkappa(x(s_+)) - \varphi^\varkappa&(x(s)) + \int_{s}^{s_+}L(x(t),v(t))dt \\
					&\geq \int_s^{s_+} \Big(-K|v(t)| + L(x(t),v(t))\Big)dt.
				\end{split}
			\end{equation}
			
			The Lagrangian can be decomposed as:
			\[L(x(t),v(t)) = L(x(t),v(t))\mathbbm{1}_{|v(t)|\geq B} + L(x(t),v(t))\mathbbm{1}_{|v(t)|<B}.\]
			
			From the choice of $B$ (see~\eqref{lower:ineq:B}), we obtain:
			\[L(x(t),v(t))\mathbbm{1}_{|v(t)|\geq B} \geq (K+1)|v(t)|\mathbbm{1}_{|v(t)|\geq B}.\]
			
			Moreover, by the definition of $C_2'$:
			\[L(x(t),v(t))\mathbbm{1}_{|v(t)|<B} \geq C_2'\mathbbm{1}_{|v(t)|<B} \geq C_2'.\]
			
			Substituting these bounds into~\eqref{lower:ineq:s_s_plus_1} gives the inequality
			\begin{equation*}
				\begin{split}
					\varphi^\varkappa(x(&s_+)) - \varphi^\varkappa(x(s)) + \int_{s}^{s_+}L(x(t),v(t))dt \\
					&\geq \int_s^{s_+} \Big(-K|v(t)| + (K+1)|v(t)|\mathbbm{1}_{|v(t)|\geq B}\Big)dt + C_2'(s_+-s).
				\end{split}
			\end{equation*}
			
			This combined with assumption~\eqref{lower:ineq:prenorm} yields
			\[\varphi^\varkappa(x(s_+)) - \varphi^\varkappa(x(s)) + \int_{s}^{s_+}L(x(t),v(t))dt \geq -\overline{H}(s_+-s).\]
			
			Substituted this estimate along with~\eqref{lower:ineq:s_0} into inequality~\eqref{lower:ineq:s_plus_0_full}, we prove the inclusion $s_+\in\mathscr{S}_{\varepsilon,\varkappa}$.
			
			\item No  $\tau$ satisfying~\eqref{lower:ineq:prenorm} exists. Then, for all $\tau\in (0,r-s]$, one has that
			\begin{equation}\label{lower:ineq:norm_estimate}
				\int_s^{s+\tau}|v(t)|dt\leq C_3'\tau,
			\end{equation}
			where $C_3'\triangleq -\overline{H}-C_2'+(K+1)B$.
			Indeed, we can decompose the integral as follows:
			\begin{equation}\label{lower:ineq:norm_L_1_v}\begin{split}
					\int_s^{s+\tau}|v(t)|dt &= \int_s^{s+\tau}|v(t)|\mathbbm{1}_{|v(t)|\geq B}dt + \int_s^{s+\tau}|v(t)|\mathbbm{1}_{|v(t)|<B}dt \\
					&\leq \int_s^{s+\tau}|v(t)|\mathbbm{1}_{|v(t)|\geq B}dt + B\tau.
				\end{split}
			\end{equation}
			
			Since assumption~\eqref{lower:ineq:prenorm} fails to hold, we can bound the first term as follows:
			\[\begin{split}
				\int_s^{s+\tau}|v(t)|\mathbbm{1}_{|v(t)|\geq B}dt &\leq \big(-\overline{H} - C_2'\big)\tau + \int_s^{s+\tau}K|v(t)|\mathbbm{1}_{|v(t)|<B}dt \\
				&\leq \big(-\overline{H} - C_2' + KB\big)\tau.
			\end{split}\]
			
			Combining this with~\eqref{lower:ineq:norm_L_1_v} yields the desired estimate~\eqref{lower:ineq:norm_estimate}.

			We choose $\tau$ to satisfy the following assumptions:
			\begin{equation}\label{lower:ineq:tau_choice}
				\frac{(C_3')^2\tau}{2\varkappa}\leq\varepsilon/3,\ \ \omega_A(C_3'\tau)\leq\varepsilon/3.
			\end{equation}
			Set $s_+=s+\tau$. Let us  introduce the following notation:
			\[\hat{b}\triangleq b^\varkappa[x(s)],\quad \hat{y}\triangleq x(s)+\hat{b},\quad \hat{p}\triangleq p^\varkappa[x(s)].\]
			Recall that $x(s_+) = x(s) + \int_s^{s_+}v(t)dt$.
			
			Applying inequality~\eqref{lower:ineq:phi_kappa_sup} yields
			\[\varphi^\varkappa(x(s_+)) - \varphi^\varkappa(x(s)) \geq \int_s^{s_+}\hat{p} v(t)dt - \frac{1}{2\varkappa}\Bigg|\int_{s}^{s_+}v(t)dt\Bigg|^2.\]
			Consequently,
			\begin{equation}\label{lower:ineq:phi_superdiff}
				\begin{split}	
					\varphi^\varkappa(x(s_+)) - &\varphi^\varkappa(x(s)) + \int_s^{s_+}L(x(t),v(t))dt \\
					\geq \int_s^{s_+}&\Big(\hat{p} v(t) + L(x(t),v(t))\Big)\mathbbm{1}_{|v(t)|< A}dt \\
					&+ \int_s^{s_+}\Big(\hat{p} v(t) + L(x(t),v(t))\Big)\mathbbm{1}_{|v(t)|\geq A}dt - \frac{(C_3')^2(s_+-s)^2}{2\varkappa}.
				\end{split}
			\end{equation}
			Here we have used inequality~\eqref{lower:ineq:norm_estimate} to bound $|\int_{s}^{s_+}v(t)dt|^2$ by $(C_3')^2(s_+-s)^2$.
			
			For $|v(t)| < A$, the definition of $\omega_A(\cdot)$, property~\ref{lower:prop:b_bound}, and the estimate $|\int_{s}^{s_+}v(t)dt| \leq C_3'\tau$ imply that
			\[L(x(t),v(t)) \geq L(\hat{y},v(t)) - \omega_A(C_3'\tau) - \omega_A(C_1\varkappa).\]
			Therefore,
			\[
			\begin{split}
				\Big(\hat{p} v(t) + &L(x(t),v(t))\Big)\mathbbm{1}_{|v(t)|< A} \\
				\geq &\Big(\hat{p} v(t) + L(\hat{y},v(t))\Big)\mathbbm{1}_{|v(t)|< A} 
				- \Big[\omega_A(C_3'\tau) + \omega_A(C_1\varkappa)\Big]\mathbbm{1}_{|v(t)|< A}.
			\end{split}
			\]
			
			Since $\hat{p}\in \partial_P^+\varphi(\hat{y})$ (see property~\ref{lower:prop:superdiff}) and $(\varphi,\overline{H})$ is a subsolution to~\eqref{intrd:eq:HJ}, we obtain the following:
			\begin{equation}\label{lower:ineq:p_L_leq_A}
				\Big(\hat{p} v(t) + L(x(t),v(t))\Big)\mathbbm{1}_{|v(t)|< A} \geq -\big(\overline{H}+\tfrac{2}{3}\varepsilon\big)\mathbbm{1}_{|v(t)|< A},
			\end{equation}
			where we have also used inequalities~\eqref{lower:ineq:varkappa_choice} and~\eqref{lower:ineq:tau_choice}.
			
			Furthermore, for $|v(t)| \geq A$, the choice of $A$ (see~\eqref{lower:ineq:A}) gives that
			\[
			\begin{split}
				\hat{p} v(t) + L(x(t),v(t)) \geq -K|v(t)| + L(x(t),v(t)) \geq -\overline{H}.
			\end{split}
			\]
			
			Substituting these inequalities into~\eqref{lower:ineq:phi_superdiff} yields:
			\[\begin{split}
				\varphi^\varkappa(x(s_+)) - &\varphi^\varkappa(x(s)) + \int_s^{s_+}L(x(t),v(t))dt \\
				&\geq \Big(-\overline{H} - \tfrac{2}{3}\varepsilon - \tfrac{(C_3')^2(s_+-s)}{2\varkappa}\Big)(s_+-s).
			\end{split}\]
			
			By our choice of $\tau$ in~\eqref{lower:ineq:tau_choice}, we conclude that
			\[\varphi^\varkappa(x(s_+)) - \varphi^\varkappa(x(s)) + \int_s^{s_+}L(x(t),v(t)) \geq -(\overline{H}+\varepsilon)(s_+-s).\]
			
			Since $s\in\mathscr{S}_{\varepsilon,\varkappa}$, this implies the following inequality
			\[\varphi^\varkappa(x(s_+)) - \varphi^\varkappa(x(0)) + \int_0^{s_+}L(x(t),v(t)) \geq -(\overline{H}+\varepsilon)s_+.\]
			
			Thus, we have shown that $s_+ \in \mathscr{S}_{\varepsilon,\varkappa}$ in this case as well.
		\end{enumerate}
		
		\item From the above, $\sup\mathscr{S}_{\varepsilon,\varkappa}=r$. Using property (K3), the inequality $\varphi^\varkappa(x(0))\geq \varphi(x(0))$ and choice of parameter condition \eqref{lower:ineq:varkappa_choice}, we conclude that
		\[\varphi(x(r))+\int_0^rL(x(t),v(t))\geq \varphi(x(0))-\overline{H}r-\varepsilon r-\varepsilon.\]
		Since $\varepsilon>0$ is arbitrary, the statement of the theorem follows.
	\end{enumerate}
\end{proof}

\section{Weak KAM theorem}\label{sect:weak_KAM}
Following \cite{Fathi2012,Sorrentino,Evans2008}, we define the Lax-Oleinik semigroup on $C(\td)$ by the following rule:
\[\begin{split}
	T^+_r\phi(y)\triangleq \inf\Bigg\{\phi(x(r))+\int_0^r&L(x(t),v(t))dt:\\(x&(t),v(t))\text{ is a controlled process, }x(0)=y\Bigg\}.\end{split}\]
	
   The main result of the paper is the following variant of the weak KAM theorem for a nonsmooth Lagrangian.

\begin{theorem}\label{weak:th:main}
	Let $(\varphi,\overline{H})$ be a viscosity solution to \eqref{intrd:eq:HJ}. Then, for every $r>0$:
	\[T^+_r\varphi=\varphi-r\overline{H}.\]
	Moreover, given $\varepsilon>0$, there exists $\varkappa_0>0$ such that, for $\varkappa\in (0,\varkappa_0]$,
	\[\limsup_{d(\Delta)\downarrow 0}\Bigg[\phi(x_{y,\varkappa,r,\Delta}(r))+\int_0^rL(x_{y,\varkappa,r,\Delta}(t),v_{y,\varkappa,r,\Delta}(t))dt\Bigg]\leq \varphi(y)-\overline{H}r+(r+1)\varepsilon.\]
\end{theorem}
This result follows immediately from Theorems~\ref{feedback:th:main} and \ref{lower:th:main} combined with the Lipschitz continuity of $\varphi$ as a viscosity solution to equation~\eqref{intrd:eq:HJ}.

The following result is also valid.

\begin{corollary}\label{weak:corollary:averaged}
	Let $(\varphi,\overline{H})$ be a viscosity solution to \eqref{intrd:eq:HJ}. Then, for every $y\in\td$:
	\[\begin{split}\lim_{r\rightarrow+\infty}\inf\Bigg\{\frac{1}{r}\int_0^rL(x(t),v&(t))dt:\\ (x&(t),v(t))\text{ is a controlled process, }x(0)=y \Bigg\}=-\overline{H}.\end{split}\]
	Furthermore, given $\varepsilon>0$, there exists $\varkappa_0>0$ such that, for every $\varkappa\in (0,\varkappa_0]$ and $y\in\td$,
	\[\lim_{r\rightarrow+\infty}\limsup_{d(\Delta)\downarrow 0}\frac{1}{r}\int_0^rL(x_{y,\varkappa,r,\Delta}(t),v_{y,\varkappa,r,\Delta}(t))dt\leq -\overline{H}+\varepsilon.\]
\end{corollary}

\begin{remark}
	This establishes the near-optimality of the feedback strategy $\mathbbm{v}_\varkappa$ defined in \eqref{feedback:intro:v_varkappa}.
\end{remark}

\section{Mather measure}\label{sect:mather}
In this section, we work with probability measures on $\TdR$, i.e., nonnegative measures on $\TdR$ normalized to 1.

\begin{definition}\label{mather:def:holonomic}
	A probability measure $\nu$ on $\TdR$ is called holonomic provided that, for every $\phi\in C^1(\td)$, the following equality holds true:
	\[\int_{\TdR}\nabla\phi(x)v\nu(d(x,v))=0.\]
\end{definition}

Hereinafter, $\nabla\phi(x)$ stands for the gradient of the function $\phi$ at the point $x\in\td$. We assume that $\nabla\phi(x)\in\rds$.

Our primary focus will be on minimizing the functional $\nu\mapsto \int_{\TdR}L(x,v)\nu(d(x,v))$ over the class of all holonomic probability measures.

\begin{definition}\label{mather:def:mather} 
	A Mather measure is a holonomic measure $\mu$ satisfying
	\[\begin{split}
		\int_{\TdR}L(x&,v)\mu(d(x,v)) \\
		&= \inf\Bigg\{\int_{\TdR}L(x,v)\nu(d(x,v)) : \nu \text{ is a holonomic measure}\Bigg\}.
	\end{split}\]
\end{definition}

\begin{theorem}\label{mather:th:mather}
	The Mather measure $\mu$ exists. Moreover, if $H$ is a Ma\~n\'e critical value, then
	\[\int_{\TdR}L(x,v)\mu(d(x,v))=-\overline{H}.\]
\end{theorem}

The proof of this result relies on the following lemma, whose proof we include for completeness, largely following the approach in \cite[p.72]{Tran}.

\begin{lemma}\label{mather:lm:subsol} 
	Let $(\varphi,\overline{H})$ be a subsolution to equation~\eqref{intrd:eq:HJ}, where $\varphi$ is Lipschitz continuous with constant $K$. Let $\eta\in C^\infty(\rd)$ be a smoothing kernel (i.e., $\eta$ vanishes outside $\mathbbm{B}_1$ and satisfies $\int_{\mathbbm{B}_1}\eta(v)dv=1$). Define the regularized function
	\[\varphi^{(\delta)}(x)\triangleq (\varphi*\eta_\delta)(x)=\int_{\mathbbm{B}_1}\varphi(x-w)\eta_\delta(w)dw,\]
	where $\eta_\delta(w)\triangleq \delta^{-d}\eta(w/\delta)$. Then, the pair $(\varphi^{(\delta)},\overline{H}+\omega_{C_3}(\delta))$ is a subsolution to equation~\eqref{intrd:eq:HJ} with $C_3$ being the constant introduced in Proposition~\ref{feedback:prop:bound}.
\end{lemma}

\begin{proof}
	Let $K$ denote the Lipschitz constant of the constant $\varphi$. Observe that, for all $x\in\td$, $p\in\rds$,
	\[H(x,-p)=\max_{v\in\rd}\big[(-p)v-L(x,v)\big]\geq -C_1',\]
	where $C_1'\triangleq \sup_{y\in\td}|L(y,0)|$ as previously defined.
	
	Recall that, in the proof of Proposition~\ref{feedback:prop:bound}, $C_3\geq 1$ was chosen such that $L(x,v)\geq (K+C_1'+1)$ whenever $|v|\geq C_3$. Consequently, for such $v$ and every $p$ with $|p|\leq K$, we have
	\[(-p)v-L(x,v)\leq -(|C_1'|+1)|v| < -C_1'.\]
	This guarantees the existence of a maximizer $v^0[x,p]\in\rd$ satisfying both
	\[H(x,-p)=(-p)v^0[x,p]-L(x,v^0[x,p])\]
	and $|v^0[x,p]|\leq C_3$.
	
	By Rademacher's theorem (see \cite[Theorem 3.2]{Evans}), the Lipschitz continuity of $\varphi$ implies its differentiability almost everywhere on $\td$. For fixed $x\in\td$, let $p\triangleq \nabla\varphi^{(\delta)}(x)$, which can be expressed as
	\[p=\int_{\mathbbm{B}_\delta}\nabla\varphi(x-w)\eta_\delta(w)dw.\]
	
	Since $(\varphi,\overline{H})$ is a subsolution to~\eqref{intrd:eq:HJ}, the inequality
	\[\big(-\nabla\varphi(x-w)\big)v^0[x,p]-L(x-w,v^0[x,p])\leq \overline{H}\]
	holds for a.e. $w\in\mathbbm{B}_\delta$. Multiplying by $\eta_\delta(w)$, integrating over $\mathbbm{B}_\delta$, and employing both the bound $|v^0[x,p]|\leq C_3$ and the definition of $\omega_{C_3}$, we obtain
	\[\big(-\nabla\varphi^{\delta}(x)\big)v^0[x,p]-L(x,v^0[x,p])\leq \overline{H}+\omega_{C_3}(\delta).\]
	The lemma's conclusion follows directly from this inequality and the characterization of $v^0[x,p]$.
\end{proof}

\begin{proof}[Proof of Theorem~\ref{mather:th:mather}] 
	We first establish that, for each holonomic measure $\nu$, the inequality
	\begin{equation}\label{mather:ineq:int_L_lower}
		\int_{\TdR}L(x,v)\nu(d(x,v)) \geq -\overline{H}
	\end{equation} 
	holds. Without loss of generality, we may assume $\int_{\TdR}L(x,v)\nu(d(x,v)) < \infty$. This, by the superlinearity of $L$, implies $\int_{\TdR}|v|\nu(d(x,v)) < \infty$.
	
	Fix $\varepsilon > 0$ and choose $A > 0$ sufficiently large so that:
	\[K\int_{\TdR}\mathbbm{1}_{|v|\geq A}|v|\nu(d(x,v)) \leq \varepsilon,\]
	and $L(x,v) \geq 0$ for all $|v| > A$. Let $\omega_A(\cdot)$ denote, as before, the modulus of continuity of $L$ with respect to its first argument on $\td\times\mathbbm{B}_A$. Furthermore, we select $\delta > 0$ satisfying \[\omega_{C_3}(\delta) \leq \varepsilon.\]
	
	For every $x \in \td$ and $v \in \rd$, consider the linear trajectory $t \mapsto x + tv$. Lemma~\ref{mather:lm:subsol} and Theorem~\ref{lower:th:main} yield the inequality
	\begin{equation}\label{mather:ineq:phi_delta_L}
		\varphi^{(\delta)}(x+\tau v) - \varphi^{(\delta)}(x) + \int_0^\tau L(x+tv,v)dt \geq -(\overline{H}+\omega_{C_3}(\delta))\tau.
	\end{equation}
	
	Observe that
	\[\left|\varphi^{(\delta)}(x+\tau v) - \varphi^{(\delta)}(x) - \nabla\varphi^{(\delta)}(x)v\right| \leq \tau|v|\xi^{(\delta)}(\tau|v|),\]
	where $\xi^{(\delta)}(\epsilon) \to 0$ as $\epsilon \to 0$ (though not necessarily uniformly in $\delta$). Moreover, for $|v| \leq A$, one has that
	\[|L(x+tv,v) - L(x,v)| \leq \omega_A(A\tau).\]
	
	For $|v| \leq A$, dividing \eqref{mather:ineq:phi_delta_L} by $\tau$ and letting $\tau \to 0$, we arrive at the inequality:
	\begin{equation*}
		\nabla\varphi^{(\delta)}(x)v + L(x,v) \geq -(\overline{H}+\omega_{C_3}(\delta)) \text{ for all } x\in\td, |v|\leq A.
	\end{equation*}
	
	Integrating over $\nu$ yields
	\begin{equation}\label{mather:ineq:L_nabla_phi}
		\begin{split}
			\int_{\TdR}\mathbbm{1}_{|v|\leq A}\nabla\varphi^{(\delta)}(x)v\nu(d(x,v)) + \int_{\TdR}\mathbbm{1}_{|v|\leq A}&L(x,v)\nu(d(x,v)) \\
			&\geq -(\overline{H}+\omega_{C_3}(\delta)).
		\end{split}
	\end{equation}
	
	Furthermore, recall that $\omega_{C_3}(\delta) \leq \varepsilon$, $|\nabla\varphi^{(\delta)}(x)| \leq K$ (this is due to Lipschitz continuity of the function $\varphi$),  $L(x,v) \geq 0$ for $|v| > A$, whereas $K\int_{|v|>A}|v|\nu(d(x,v))\leq\varepsilon$. This and inequality~\eqref{mather:ineq:L_nabla_phi} yield the following lower bound:
	\[
	\int_{\TdR}\nabla\varphi^{(\delta)}(x)v\nu(d(x,v)) + \int_{\TdR}L(x,v)\nu(d(x,v)) \geq -\overline{H} - 2\varepsilon.
	\]
	
	The holonomic property of $\nu$ then implies that, for every $\varepsilon > 0$:
	\[\int_{\TdR}L(x,v)\nu(d(x,v)) \geq -\overline{H} - 2\varepsilon,\]
	which establishes \eqref{mather:ineq:int_L_lower}.
	
	To construct a holonomic measure $\mu$ satisfying
	\begin{equation}\label{mather:ineq:L_upper}
		\int_{\TdR}L(x,v)\mu(d(x,v)) \leq -\overline{H},
	\end{equation}
	we proceed as follows. Fix $y \in \td$ and apply Corollary~\ref{weak:corollary:averaged}. For each $N \in \mathbb{N}$, there exist a parameter $\varkappa_N$ and a partition $\Delta_N$ of $[0,N]$ such that the controlled process:
	\[(\hat{x}_N(\cdot),\hat{v}_N(\cdot)) = (x_{y,\varkappa_N,N,\Delta_N}(\cdot), v_{y,\varkappa_N,N,\Delta_N}(\cdot))\]
	satisfies
	\begin{equation}\label{mather:ineq:int_L_N}
		\frac{1}{N}\int_0^N L(\hat{x}_N(t),\hat{v}_N(t))dt \leq -\overline{H} + \frac{1}{N}.
	\end{equation}
	
	By Proposition~\ref{feedback:prop:bound}, we have that, for every $t\in [0,T]$,
	\begin{equation}\label{mather:ineq:v_hat}
		|\hat{v}_N(t)| \leq C_3.
	\end{equation}
	
	Let us define probability measures $\mu_N$ on $\TdR$ as follows:
	\[\int_{\TdR}\psi(x,v)\mu_N(d(x,v)) \triangleq \frac{1}{N}\int_0^N \psi(\hat{x}_N(t),\hat{v}_N(t))dt \quad \forall \psi \in C_b(\TdR).\]
	Due to \eqref{mather:ineq:v_hat}, each measure $\mu_N$ is supported on $\td\times\mathbbm{B}_{C_3}$. By \cite[Theorem 5.1]{Billingsley1999}, there exists a subsequence $\{\mu_{N_k}\}$ converging narrowly to some probability measure $\mu$. The latter is concentrated on $\td\times\mathbbm{B}_{C_3}$.
	
	To verify that $\mu$ is holonomic, consider any $\phi \in C^1(\td)$. We have that
	\[\phi(\hat{x}_{N_k}(N_k)) - \phi(\hat{x}_{N_k}(0)) = \int_0^{N_k} \nabla\phi(\hat{x}_{N_k}(t))\hat{v}_{N_k}(t)dt.\]
	
	Using the definition of the measure $\mu_N$ with a function $\psi\in C_b(\td\times\rd)$ such that  $\psi(x,v) = \nabla\phi(x)v$ when $|v|\leq C_3$, we deduce that
	\[\frac{1}{N_k}\left[\phi(\hat{x}_{N_k}(N_k)) - \phi(\hat{x}_{N_k}(0))\right] = \int_{\TdR} \nabla\phi(x)v \mu_{N_k}(d(x,v)).\]
	Taking $k \to \infty$ establishes the holonomic property for $\mu$.
	
	Finally, from \eqref{mather:ineq:int_L_N} we have the inequality:
	\[\int_{\TdR} L(x,v) \mu_{N_k}(d(x,v)) \leq -\overline{H} + \frac{1}{N_k}.\]
	Letting $k \to \infty$ yields \eqref{mather:ineq:L_upper}. This completing the proof.
\end{proof}

\bibliography{KAM_proximal.bib}
\end{document}